\documentclass{article}
\usepackage[english]{babel}
\usepackage{epsfig}
\usepackage{amsmath}
\usepackage{amsthm}
\usepackage{amssymb}

\newtheorem{theorem}{Theorem}[section]
\newtheorem{definition}[theorem]{Definition}
\newtheorem{lemma}[theorem]{Lemma}
\newtheorem{proposition}[theorem]{Proposition}
\newtheorem{corollary}[theorem]{Corollary}
\newtheorem{remark}[theorem]{Remark}
\newtheorem{example}[theorem]{Example}

\newcommand{\hh}{{\mathbb{H}}}
\newcommand{\s}{{\mathbb{S}}}
\newcommand{\cc}{{\mathbb{C}}}
\newcommand{\rr}{{\mathbb{R}}}
\newcommand{\nn}{{\mathbb{N}}}
\newcommand{\zz}{{\mathbb{Z}}}

\title{\bf A Phragm\'en - Lindel\"of principle for slice regular functions}
\author{Graziano Gentili  \\ 
\normalsize Dipartimento di Matematica ``U. Dini'', Universit\`a di Firenze \\ 
\normalsize Viale Morgagni 67/A, 50134 Firenze, Italy,  gentili@math.unifi.it \\
\and Caterina Stoppato \\ 
\normalsize Dipartimento di Matematica ``U. Dini'', Universit\`a di Firenze \\ 
\normalsize Viale Morgagni 67/A, 50134 Firenze, Italy,  stoppato@math.unifi.it\\
\and Daniele C. Struppa \\ 
\normalsize Department of Mathematics and Computer Science\\ 
\normalsize Schmid College of Science, Chapman University\\  
\normalsize One University Drive, Orange, CA 92866 USA, struppa@chapman.edu}

\date{  }
\begin{document}
\maketitle

\section{Introduction}

The celebrated $100$-year old Phragm\'en-Lindel\"of theorem, \cite{phragmen,phragmen2}, is a far reaching extension of the maximum modulus theorem for holomorphic functions that in its simplest form can be stated as follows:

\begin{theorem}\label{classiccomplex}
Let $\Omega \subset \cc$ be a simply connected domain whose boundary contains the point at infinity. If $f$ is a bounded holomorphic function on $\Omega$ and $\limsup_{z \to z_0}|f(z)| \leq M$ at each finite boundary point $z_0$, then $|f(z)| \leq M$ for all $z \in \Omega$.
\end{theorem}

The term Phragm\'en-Lindel\"of also applies to a number of variations of this result, which guarantee a bound for holomorphic functions, when conditions are known on their growth. The two most famous variations deal with functions which are holomorphic in an angle or in a strip, and can be stated as follows (see, for instance, \cite{conway,levin} as well as \cite{ahlfors, heins}).

\begin{theorem}\label{angle}
Let $f$ be a holomorphic function on an angle $\Omega$ of opening $\frac{\pi}{\alpha}$. Suppose $f$ is continuous up to the boundary and such that, for some  $\rho<\alpha$, $|f(z)|\leq \exp(|z|^{\rho})$ asymptotically. If there exists an $M\geq 0$ such that $|f|\leq M$ in $\partial \Omega$ then $|f|\leq M$ in $\Omega$. 
\end{theorem}

\begin{theorem}\label{strip}
Let $f$ be a holomorphic function on a strip $\Omega$ of width $2 \gamma$, continuous up to the boundary. Suppose that $|f(z)| \leq N \exp(e^{k|z|})$ in $\Omega$ for some positive constants $N$ and $k<\frac{\pi}{2\gamma}$. If there exists an $M\geq 0$ such that $|f|\leq M$ in $\partial \Omega$ then $|f|\leq M$ in $\Omega$. 
\end{theorem}

In some recent papers \cite{jin1,jin2,kheyfits2,kheyfits} there has been a resurgence of interest in Phragm\'en-Lindel\"of type theorems. Specifically, \cite{jin1,jin2} consider solutions of suitable partial differential equations while \cite{kheyfits2,kheyfits} deal with the case of functions of a hypercomplex variable. In the present article we obtain the analogs of theorems \ref{classiccomplex}, \ref{angle} and \ref{strip} for slice regular functions, a class of functions of a quaternionic variable introduced in \cite{cras,advances} and studied in subsequent papers (for a survey, see \cite{survey}).

In section \ref{preliminaries} we will provide the necessary background about the theory of slice regular quaternionic functions. In section \ref{phragmen} we give direct proofs of quaternionic analogs of theorems  \ref{classiccomplex} and \ref{angle}. Finally, in section \ref{slicewise} we use a different approach, which exploits the intrinsic nature of slice regular functions, to extend our results.

\emph{Acknowledgements.} The first two authors are grateful to Chapman University for the hospitality during the preparation of this paper.

\section{Preliminaries}\label{preliminaries}

Let $\mathbb{H}$ denote the skew field of quaternions. Its elements are of the form $q=x_0+ix_1+jx_2+kx_3$ where the $x_l$ are real, and $i$, $j$, $k$ are such that
$$ i^2 = j^2 = k^2 = -1,$$
$$ ij=-ji=k,\ \  jk=-kj=i,\ \ ki=-ik=j.$$
We set
$$Re(q)=x_0,\qquad Im(q) =i x_1 +j x_2 +k x_3,\qquad |q|=\sqrt{x_0^2+x_1^2+x_2^2+x_3^2}.$$
$Re(q)$, $Im(q)$ and $|q|$ are called the \emph{real part}, the \emph{imaginary part} and the \emph{module} of $q$, respectively. The quaternion 
$$\bar q=Re(q)-Im(q)=x_0-i x_1-j x_2-k x_3$$
is called the {\em conjugate} of $q$ and satisfies
$$|q|=\sqrt{q\bar q}=\sqrt{\bar q q}.$$
The {\em inverse} of any element $ q\not= 0$ is given by
$$q^{-1}={{\bar q}\over{|q|^2}}.$$

We denote by $\mathbb{S}$ the unit sphere of purely imaginary quaternions, i.e.
$$\mathbb{S}=\{q=ix_1+jx_2+kx_3: x_1^2+x_2^2+x_3^2=1\}$$
so that every quaternion $q$ which is not real (i.e. with $Im(q) \neq 0$) can be written as $q = x+Iy$ for $x = Re(q), y=|Im(q)|$ and $I = \frac{Im(q)}{|Im(q)|} \in \s$. Also, if we set $r = |q|$ then $q = r e^{I \vartheta}$ for some $I \in \s$ and $\vartheta \in \rr$.

Beginning with the seminal papers of Fueter \cite{fueter1,fueter2}, many mathematicians have developed theories of holomorphicity in the quaternionic setting (for an overview, see the introduction of \cite{survey}). More recently, in \cite{cras, advances}, the authors proposed a new notion of holomorphicity (called \emph{slice regularity}) for quaternion-valued functions of a quaternionic variable. Unlike Fueter's, this theory includes the polynomials and the power series of the quaternionic variable $q$ of the type $\sum_{n \geq 0} q^n a_n$, with coefficients $a_n \in \hh$. Furthermore, analogs of most of the fundamental properties of holomorphic functions of one complex variable can be proven in this new setting (see also \cite{survey} and references therein).

\begin{definition}\label{regularity} 
Let $\Omega$ be an open set in $\mathbb{H}$. A function $f:\Omega \to \mathbb{H}$ is said to be \emph{slice regular} if, for every $I \in \mathbb{S}$, its restriction $f_I$ to the complex line $L_I=\mathbb{R}+\mathbb{R}I$ passing through the origin and containing $1$ and $I$ satisfies
$$\overline{\partial}_If(x+yI):=\frac{1}{2}\left(\frac{\partial}{\partial x}+I\frac{\partial}{\partial y}\right)f_I(x+yI)=0,$$
in $\Omega_I = \Omega \cap L_I$.
\end{definition}

The following result is a key tool in the study of slice regular functions, and it will be used extensively in section \ref{slicewise}.

\begin{lemma}[Splitting Lemma]\label{splitting}  
If $f$ is a slice regular function on an open set $\Omega$ then, for every $I \in \mathbb{S}$ and every $J \perp I$ in $\mathbb{S}$, there exist two holomorphic functions $F,G : \Omega_I \to L_I$ such that
$$f_I(z)=F(z)+G(z)J$$
for all $z\in \Omega_I$.
\end{lemma}

We now identify a class of domains that naturally qualify as domains of definition of regular functions.

\begin{definition}
Let $\Omega$ be a domain in $\mathbb{H}$. We say that $\Omega$ is a \emph{slice domain} if $\Omega \cap \mathbb{R}$ is non empty and if $\Omega_I = \Omega \cap L_I$ is a domain in $L_I$ for all $I \in \mathbb{S}$.
\end{definition}

Indeed, an analog of the identity principle holds for slice regular functions on slice domains.

\begin{theorem}\label{identity}
Let $f,g:\Omega\to \mathbb{H}$ be slice regular functions on a slice domain $\Omega$. If $f$ and $g$ coincide in $T \subseteq \Omega$ and if there exists $I \in \s$ such that $T_I = T\cap L_I$ has an accumulation point in $\Omega_I$, then $f$ and $g$ coincide in $\Omega$.
\end{theorem}

Furthermore, analogs of the two classic statements of the maximum modulus principle hold.

\begin{theorem}\label{weakmaximum} 
Let $f:\Omega\to \mathbb{H}$ be a slice regular function on a slice domain $\Omega$. If $|f|$ has a relative maximum point in $\Omega$, then $f$ is constant in $\Omega$.
\end{theorem}

\begin{corollary}\label{maximum}
Let $f:\Omega\to \mathbb{H}$ be a slice regular function on a bounded slice domain $\Omega$. If $\limsup_{q \to q_0} |f(q)|\leq M$ for all $q_0 \in \partial \Omega$ then $|f|\leq M$ in $\Omega$.
\end{corollary}

Due to the non-commutativity of $\hh$, pointwise multiplication and composition do not preserve slice regularity in general. Nevertheless, slice regularity is preserved for the following class of functions.

\begin{definition}
Let $f:\Omega\to \mathbb{H}$ be a slice regular function. We say that $f$ is a \emph{slice preserving function} if $f(\Omega_I) \subseteq L_I$ for all $I \in \mathbb{S}$.
\end{definition}

\begin{proposition}\label{multformula}
Let $f,g:\Omega\to \mathbb{H}$ be slice regular functions. If $f$ is a slice preserving function then the product $f\cdot g$ is slice regular.
\end{proposition}

\begin{proposition}\label{composition}
Let $f:\Omega\to \Omega' \subseteq \mathbb{H}$ and $g : \Omega' \to \hh$ be slice regular functions. If $f$ is a slice preserving function then the composition $g \circ f$ is slice regular.
\end{proposition}

\section{The Phragm\'en - Lindel\"of principle}\label{phragmen}

In this section we will give a direct proof of the Phragm\'en-Lindel\"of principle for slice regular functions defined on suitable domains $\Omega$ in the quaternionic space $\hh$. We will also study the special case in which $\Omega$ is a cone.

As in the complex case, the quaternionic Phragm\'en-Lindel\"of principle generalizes the maximum modulus principle \ref{maximum} to unbounded domains. Let $\widehat{\hh} = \hh \cup \{\infty\}$ denote the Alexandroff compactification of $\hh$. We define the \emph{extended boundary} $\partial_\infty \Omega$ of any $\Omega \subseteq \widehat{\hh}$ to be the boundary of the closure of $\Omega$ in $\widehat{\hh}$. As customary, we will denote by $\partial \Omega= \partial_{\infty} \Omega \setminus \{\infty\}$ the finite boundary of $\Omega$.

\begin{theorem}[Phragm\'en-Lindel\"of principle]\label{classic}
Let $\Omega \subset \hh$ be a domain whose extended boundary contains the point at infinity and suppose that there exist a real point $t \in \rr \cap \Omega$ such that $\Omega \setminus (-\infty, t]$ (or $\Omega \setminus [t, +\infty)$) is a slice domain. If $f$ is a bounded slice regular function on $\Omega$ and $\limsup_{q \to q_0}|f(q)| \leq M$ for all $q_0 \in \partial \Omega$, then $|f(q)| \leq M$ for all $q \in \Omega$.
\end{theorem}

\begin{proof}
Since $q \mapsto q+t$ and $q \mapsto -q$ are slice preserving functions, by proposition \ref{composition} we can assume that $t = 0$ and  $\Omega \setminus (-\infty, 0]$ is a slice domain. Choose $r>0$ such that the closure of $B_r = B(0,r)$ is contained in $\Omega$ and let $\omega_r(q) = q^{-1}r$ for $q \neq 0$. Notice that $|\omega_r|<1$ in $\hh \setminus \overline{B_r}$, that $|\omega_r| = 1$ on $\partial B_r$, and that $\omega_r$ is a slice preserving regular function.

For all $q \in \hh \setminus (-\infty,0]$, define the principal logarithm of $q$ as 
$$Log(q) =  \ln |q| + \arccos\left(\frac{Re(q)}{|q|}\right) \frac{Im(q)}{|Im(q)|}.$$ 
Notice that the principal logarithm is a slice regular function and it is slice preserving. By proposition \ref{composition}, setting $\omega_r^\delta(q) := e^{\delta Log \omega(q)}$  for all $q \in \hh \setminus (-\infty,0]$ defines a slice regular function. Finally, by proposition \ref{multformula} the product $\omega_r^\delta f$ is a slice regular function on $\Omega' = \Omega \setminus (\overline{B_r} \cup (-\infty,r])$, which by hypothesis is a slice domain when $r$ is sufficiently small. The behavior of $|\omega_r^{\delta}f|$ on the extended boundary $\partial_\infty \Omega' = \{\infty\} \cup \partial \Omega \cup \partial B_r \cup (\Omega \cap (-\infty,r])$ is the following:
\begin{enumerate}
\item $\limsup_{q \to \infty} |\omega_r^{\delta}f(q)| = \limsup_{q \to \infty}|f(q)| \frac{r^\delta}{|q|^\delta} = 0,$
\item $\limsup_{q \to q_0} |\omega_r^{\delta}f(q)| 
< \limsup_{q \to q_0}|f(q)| \leq M$ for all $q_0 \in \partial \Omega,$
\item $\limsup_{q \to q_0} |\omega_r^{\delta}f(q)| 
= |f(q_0)| \leq \max_{\partial B_r}|f|=: M_r$ for all $q_0 \in \partial B_r,$
\item $\limsup_{q \to q_0} |\omega_r^{\delta}f(q)| \leq \sup_{\Omega \cap (-\infty,r]} \frac{r^\delta}{|q|^\delta}|f(q)| =: N$  for all $q_0 \in \Omega \cap (-\infty,r].$
\end{enumerate}

Let us prove that $N$ is finite. Choose $\{a_n\}_{n \in \nn} \subset \overline{\Omega} \cap (-\infty,r]$ such that $\lim_{n \to \infty}\frac{r^\delta}{|a_n|^\delta}|f(a_n)| = N$. If $N = 0$, there is nothing to prove. Otherwise,  by point 1, $\{a_n\}_{n \in \nn}$ must be bounded. By possibly extracting a subsequence, we may suppose $\{a_n\}_{n \in \nn}$ to converge to some $q_0 \in \overline{\Omega} \cap (-\infty,r]$. If $q_0 \in \partial \Omega$ then $N \leq M$ by hypothesis. Else $q_0 \in \Omega$ and $N = \frac{r^\delta}{|q_0|^\delta}|f(q_0)|$.

As a consequence of points 1-4, $\limsup_{q \to q_0}|\omega_r^\delta f| \leq \max\{M, M_r, N\}$ for all $q_0 \in \partial_\infty \Omega'$ and, by an easy application of the maximum modulus principle \ref{maximum}, $|\omega_r^\delta f| \leq \max\{M, M_r, N\}$ in $\Omega'$.

Now let us prove that $N \leq \max\{M, M_r\}$. Suppose by contradiction that the opposite inequality holds. In particular $N>M$ and (as we explained above) there exists $q_0 \in \Omega\cap (-\infty, r])$ such that $N = \frac{r^\delta}{|q_0|^\delta}|f(q_0)|$. In a ball $B(q_0,\varepsilon)$ contained in $\Omega \setminus \overline{B_r}$, we define a new branch of logarithm $log$ by letting 
$$log (q) = \ln |q| + \left[\arccos\left(\frac{Re(q)}{|q|}\right) - \pi\right] \frac{Im(q)}{|Im(q)|}.$$ 
As before, the function $g = e^{\delta log \omega_r} f$ is slice regular in $B(q_0,\varepsilon)$ and $|g(q)| = \frac{r^\delta}{|q|^\delta}  |f(q)|$ for all $q \in B(q_0,\varepsilon)$. As a consequence, $|g(q_0)| \geq |g(q)|$ for all $q \in B(q_0, \varepsilon)$. Indeed:
\begin{enumerate}
\item for all $q \in (q_0-\varepsilon, q_0 +\varepsilon)$, $|g(q)| \leq \sup_{\Omega \cap (-\infty,r]} \frac{r^\delta}{|q|^\delta}|f(q)| = N = |g(q_0)|$; 
\item for all $q \in B(q_0,\varepsilon) \setminus (q_0-\varepsilon, q_0 +\varepsilon)$, we proved $|g(q)| = |\omega_r^\delta(q) f(q)| = \frac{r^\delta}{|q|^\delta} |f(q)| \leq \max\{M, M_r, N\} = N=|g(q_0)|$.
\end{enumerate}
Hence $|g|$ has a maximum at $q_0$ and $g$ must be constant. Therefore $|\omega_r^\delta f| =|g|\equiv N$ in $B(q_0,\varepsilon) \setminus (q_0-\varepsilon, q_0 +\varepsilon)$. In particular $\omega_r^\delta f$, which is a slice regular function on the slice domain $\Omega'$, has an interior maximum point. As before, the maximum modulus principle \ref{maximum} yields that $\omega_r^\delta f$ must be constant. As a consequence, there exists a constant $c$ such that $f(q)=q^\delta c$ in $\Omega'$, a contradiction with the hypothesis that $f$ is bounded.

So far, we proved that $|\omega_r|^{\delta} |f| \leq \max\{M, M_r\}$
in $\Omega\setminus \overline{B_r}$. We deduce that 
$$|f| \leq \frac{\max\{M, M_r\}}{|\omega|^\delta}$$ 
in $\Omega\setminus \overline{B_r}$ and letting $\delta \to 0^+$ we conclude that $|f|\leq \max\{M, M_r\}$ in $\Omega \setminus \overline{B_r}$.

If we let $r \to 0^+$ we obtain $|f(q)| \leq \max \{M, |f(0)|\}$ for all $q \in \Omega \setminus \{0\}$, hence for all $q \in \Omega$. Finally, we prove that $|f(0)| \leq M$: if it were not so, then $|f|$ would have a maximum at $0$, a contradiction by the maximum modulus principle \ref{weakmaximum}.
 \end{proof}
 
We now tackle the case in which $\Omega$ is the \emph{circular cone} 
$$C\left(\varphi\right) = \{re^{I \vartheta} : r>0, |\vartheta| < \varphi/2, I \in \s\}$$ 
for some $\varphi < 2 \pi$. Such cones certainly satisfy the hypotheses of theorem \ref{classic}. Moreover, we can prove that it is not necessary to suppose that $f$ is bounded as long as the \emph{opening} $\varphi$ of the cone is suitably related to the growth order of $f$, defined as follows. If $f$ is a slice regular function on $\Omega = C\left(\varphi\right)$, continuous up to the boundary, we set
\begin{equation}\label{modulus}
M_f(r, \Omega) = \max \{|f(q)| : q \in \overline{\Omega}, |q| = r\}
\end{equation}
and define the \emph{order} $\rho$ of $f$ as
\begin{equation}\label{order}
\rho = \limsup_{r \to +\infty} \frac{\ln^+ \ln^+ M_f(r, \Omega)}{\ln r}.
\end{equation}

\begin{theorem}[Phragm\'en-Lindel\"of principle for circular cones]\label{cones}
Let $f$ be a slice regular function in $C\left(\frac{\pi}{\alpha}\right)$, continuous up to the boundary. Suppose the order $\rho$ of $f$ to be strictly less than $\alpha$. If there exists an $M\geq 0$ such that $|f|\leq M$ in $\partial C\left(\frac{\pi}{\alpha}\right)$ then $|f|\leq M$ in $C\left(\frac{\pi}{\alpha}\right)$. 
\end{theorem}

\begin{proof}
Choose $\gamma$ such that $\rho < \gamma<\alpha$. For $q \in C\left(\frac{\pi}{\alpha}\right)$ we define $\omega(q) = e^{-q^\gamma}$ with $q^\gamma = e^{\gamma Log(q)}$ (where $Log(q)$ is the principal logarithm of $q$). For all $\delta>0$ we set $\omega^\delta(q) = e^{-\delta q^\gamma}$ and we have that
$$|\omega^\delta(r e^{I \vartheta})| = e^{-\delta r^\gamma \cos (\gamma \vartheta)}.$$
For $-\frac{\pi}{2 \alpha}< \vartheta < \frac{\pi}{2 \alpha}$ and $\rho < \rho_1 < \gamma$ the following holds asymptotically:
$$|\omega^\delta(r e^{I \vartheta}) f(r e^{I \vartheta})| < e^{r^{\rho_1}-\delta r^\gamma \cos (\gamma \vartheta)}.$$
Since $\gamma < \alpha$, we have $-\frac{\pi}{2}<\gamma \vartheta < \frac{\pi}{2}$ so that $\cos (\gamma \vartheta)>0$. Since $\rho_1 < \gamma$ we conclude that in $C\left(\frac{\pi}{\alpha}\right)$
$$\lim_{q \to \infty} |\omega^\delta(q) f(q)| = 0.$$
Since for all $q \in \partial C\left(\frac{\pi}{\alpha}\right)$ we have $|\omega^\delta(q) f(q)| < |f(q)| \leq M$, we conclude that $\lim_{q \to q_0} |\omega^\delta(q) f(q)| \leq M$ for all $q_0 \in \partial_\infty C\left(\frac{\pi}{\alpha}\right)$. Applying the maximum modulus principle \ref{maximum}, we get $|\omega^\delta f| \leq M$ in $C\left(\frac{\pi}{\alpha}\right)$. The inequality $|f| \leq \frac{M}{|\omega|^\delta}$, which holds for all $\delta >0$, yields that $|f| \leq M$  in $C\left(\frac{\pi}{\alpha}\right)$.
\end{proof}

In the next section we will offer an alternative proof extending theorem \ref{cones} to a larger class of domains.

\section{A slicewise approach}\label{slicewise}

The proofs we gave in the previous section are intrinsic to the quaternionic setting. However, in the theory of slice regular functions it is often possible to use a different technique. Specifically (see e.g. \cite{advances}), one can apply the splitting lemma \ref{splitting} to reduce to the case of holomorphic functions of one complex variable. In this section, we employ this technique to give different proofs of theorems \ref{classic} and \ref{cones}. In fact, this allows slight variations of the hypotheses. We also prove two other results, which we were unable to obtain using a direct approach.

\begin{theorem}\label{classicslice}
Let $\Omega \subset \hh$ be a domain whose extended boundary contains the point at infinity and such that, for all $I \in \s$, $\Omega_I = \Omega \cap L_I$ is simply connected. If $f$ is a bounded slice regular function on $\Omega$ and $\limsup_{q \to q_0}|f(q)| \leq M$ for all $q_0 \in \partial \Omega$, then $|f(q)| \leq M$ for all $q \in \Omega$.
\end{theorem}

\begin{proof}
Suppose that there exists $p \in \Omega$ such that $|f(p)|>M$. Possibly multiplying $f$ by the constant $\frac{\overline{f(p)}}{|f(p)|}$, we may suppose $f(p) > 0$. Let $I \in \s$ be such that $p \in L_I$, choose $J \in \s$ such that $J \perp I$ and let $F,G: \Omega_I = \Omega \cap L_I \to L_I$ be holomorphic functions such that $f_I = F + G J$ (see lemma \ref{splitting}). Then $f_I(p) = F(p)$. On the other hand, since $|F| \leq |f_I| \leq M$ in $\partial \Omega_I \subseteq \partial \Omega$ and $|F| \leq |f_I|$ is bounded in $\Omega_I$, we must have $|F| \leq M$ in $\Omega_I$ by the complex Phragm\'en-Lindel\"of principle \ref{classiccomplex}.
\end{proof}

\begin{remark}
We required $\Omega_I$ to be simply connected, as in the classic (complex) Phragm\'en-Lindel\"of principle \ref{classiccomplex}. Notice however that this hypothesis can be weakened (see exercise 1 in \cite{conway}).
\end{remark}

A similar proof allows us to extend theorem \ref{cones} to a larger class of domains. 

\begin{definition}
We call a slice domain $\Omega \subset \hh$ an \emph{angular domain} if, for all $I \in \s$, $\Omega_I$ is an angle $\{r e^{I(\zeta_I+\vartheta)} : r>0, |\vartheta| <\varphi_I/2\}$ for some $\zeta_I,\varphi_I$ with $\zeta_I \in \rr, 0<\varphi_I<2 \pi$. The \emph{opening} of $\Omega$ is defined to be $\sup_{I \in \s} |\varphi_I|$.
\end{definition}

The following proposition shows, once again, the surprising geometrical properties of the quaternions.

\begin{proposition}
Let $\Omega$ be an open subset of $\hh$ such that, for all $I \in \s$, $\Omega_I$ is an angle $\{r e^{I(\zeta_I+\vartheta)} : r>0, |\vartheta| <\varphi_I/2\}$. If $I \mapsto \zeta_I$ and $I \mapsto \varphi_I$ are continuous in $\s$ then $\Omega$ is automatically a slice domain.
\end{proposition}

\begin{proof}
In order to prove our assertion it suffices to show that at least one slice $\Omega_I$ contains a real half line. Notice that, for all $I \in \s$, $\Omega_I = \Omega_{-I}$. Hence $\varphi_I= \varphi_{-I}$ and $\zeta_I = 2 k \pi - \zeta_{-I}$ for some $k \in \zz$. Since
$$I \mapsto (\zeta_I,\varphi_I)$$
is a continuous function from $\s \simeq S^2$ to $\rr^2$, by the Borsuk-Ulam theorem (see Corollary 9.3 in \cite{massey}) there exist two antipodal points of $\s$ having the same image. In particular, there exists a $J \in \s$ such that $\zeta_J = \zeta_{-J}$ and we conclude that $\zeta_J = k \pi$ for some $k \in \zz$.
\end{proof}

The \emph{order} of a slice regular function $f$ on an angular domain $\Omega$ is defined by equations (\ref{modulus}) and (\ref{order}) as in the case of circular cones.

\begin{theorem}\label{conesslice}
Let $f$ be a slice regular function on an angular domain $\Omega$ of opening $\frac{\pi}{\alpha}$. Suppose $f$ is continuous up to the boundary and has order $\rho<\alpha$. If there exists an $M\geq 0$ such that $|f|\leq M$ in $\partial \Omega$ then $|f|\leq M$ in $\Omega$. 
\end{theorem}

The proof of theorem \ref{conesslice} is completely analogous to that of theorem \ref{classicslice} and it makes use of the Phragm\'en-Lindel\"of principle for complex angles \ref{angle}. As in the complex case, the hypothesis that $\rho<\alpha$ cannot be weakened. Indeed, we have the following.

\begin{example}
We can define a slice regular function $f$ of order $\rho>0$ on $C\left(\frac{\pi}{\rho}\right)$ by setting $f(q) = e^{q^\rho}$ where $q^\rho = e^{\rho Log(q)}$. We notice that, for all $q = r e^{\pm I \frac{\pi}{2\rho}} \in \partial C\left(\frac{\pi}{\rho}\right)$, $|f(q)|= |\exp(r^\rho e^{\pm I \frac{\pi}{2}})| = |\exp(\pm I r^\rho)| = 1$, while the function $f$ is unbounded in $C\left(\frac{\pi}{\rho}\right)$.
\end{example}

Nevertheless, as in the complex case, when a function $f$ has order $\rho$ in angular domain of opening $\frac{\pi}{\rho}$ we can control the growth of $f$ in terms of its \emph{type} $\sigma$, defined as
\begin{equation}\label{type}
\sigma = \limsup_{r \to +\infty} \frac{\ln^+ M_f(r, \Omega)}{r^\rho}.
\end{equation}

\begin{theorem}\label{conessharp}
Let $\Omega$ be an angular domain with $\Omega_I = \{r e^{I(\zeta_I + \vartheta)} : r>0, | \vartheta|<\varphi_I/2\}$ for all $I \in \s$. Let $f$ be a slice regular function of order $\rho$ and type $\sigma$ on $\Omega$, continuous up to the boundary. If the opening of $\Omega$ is not greater than $\frac{\pi}{\rho}$ and $|f|$ is bounded by $M$ in $\partial \Omega$, then for all $I \in \s$
\begin{equation}
| f(r e^{I(\zeta_I +  \vartheta)}) | \leq M e^{\sigma r^\rho \cos(\rho \vartheta)}
\end{equation}
for $r>0$ and $|\vartheta|<\varphi_I/2$.
\end{theorem}

\begin{proof}
Suppose that there exists $p = r e^{I(\zeta_I +  \vartheta)} \in \Omega_I$ such that $$|f(p)| > M e^{\sigma r^\rho \cos(\rho\vartheta)}.$$ As in the proof of theorem \ref{classicslice}, we may suppose $f(p) > 0$. Choosing $J \in \s$ with $J \perp I$ and holomorphic functions $F,G: \Omega_I \to L_I$ such that $f_I = F + G J$, we have $f_I(p) = F(p)$. Now, $|F| \leq |f_I| \leq M$ in $\partial \Omega_I$ and $F$ has order less than or equal to $\rho$ and type less than or equal to $\sigma$ in $\Omega_I$. By theorem 22 in \cite{levin}, we conclude $|F(p)| \leq M e^{\sigma r^\rho \cos(\rho\vartheta)}$, a contradiction.
\end{proof}

An analogous proof allows us to derive the quaternionic version of theorem \ref{strip}.

\begin{definition}
We call a slice domain $\Omega \subset \hh$ a \emph{strip domain} if, for all $I \in \s$, there exist a line $\ell_I$ in $L_I$ and a positive real number $\gamma_I$ such that $\Omega_I$ is the strip $\{z \in L_I : |z - \ell_I| < \gamma_I/2\}$. The \emph{width} of $\Omega$ is defined to be $\sup_{I \in \s} |\gamma_I|$.
\end{definition}

\begin{theorem}
Let $f$ be a slice regular function on a strip domain $\Omega$ of width $\gamma$, continuous up to the boundary. Suppose that $|f(q)| \leq N \exp(e^{k|q|})$ in $\Omega$ for some positive constants $N$ and $k<\frac{\pi}{\gamma}$. If there exists an $M\geq 0$ such that $|f|\leq M$ in $\partial \Omega$ then $|f|\leq M$ in $\Omega$. 
\end{theorem}

The slicewise approach adopted in this section has an important bearing for entire functions, i.e. for slice regular functions on $\Omega = \hh$. Indeed, if we define the order and the type of the entire function $f$ by equations (\ref{modulus}), (\ref{order}) and (\ref{type}), then the quaternionic Liouville theorem proven in \cite{advances} generalizes as follows.

\begin{theorem}
Let $f$ be a quaternionic entire function of first order at most, having type $0$. In other words, for all $\varepsilon >0$ we suppose $|f(q)| < e^{\varepsilon |q|}$ when $|q|$ is large enough. If, for some $I \in \s$, the plane $L_I$ contains a line on which $|f|$ is bounded then $f$ is constant.
\end{theorem}

\begin{proof}
Let $J \in \s$ be orthogonal to $I$ and let $F,G: L_I \to L_I$ be holomorphic functions such that $f_I = F + G J$. By the corollary to theorem 22 in \cite{levin}, $F$ and $G$ are constant. Hence $f_I$ is constant and, by the identity principle \ref{identity}, $f$ must be constant too.
\end{proof}


\bibliography{Phragmen}

\begin{thebibliography}{10}

\bibitem{ahlfors}
L.~V. Ahlfors.
\newblock On {P}hragm\'en-{L}indel\"of's principle.
\newblock {\em Trans. Amer. Math. Soc.}, 41(1):1--8, 1937.

\bibitem{conway}
J.~B. Conway.
\newblock {\em Functions of one complex variable}, volume~11 of {\em Graduate
  Texts in Mathematics}, chapter VI {\S}4, pages 138--141.
\newblock Springer-Verlag, New York, second edition, 1978.

\bibitem{fueter1}
R.~Fueter.
\newblock Die {F}unktionentheorie der {D}ifferentialgleichungen {$\Theta u=0$}
  und {$\Theta\Theta u=0$} mit vier reellen {V}ariablen.
\newblock {\em Comment. Math. Helv.}, 7(1):307--330, 1934.

\bibitem{fueter2}
R.~Fueter.
\newblock \"{U}ber die analytische {D}arstellung der regul\"aren {F}unktionen
  einer {Q}uaternionenvariablen.
\newblock {\em Comment. Math. Helv.}, 8(1):371--378, 1935.

\bibitem{survey}
G.~Gentili, C.~Stoppato, D.~C. Struppa, and F.~Vlacci.
\newblock Recent developments for regular functions of a hypercomplex variable.
\newblock In I.~Sabadini, M.~Shapiro, and F.~Sommen, editors, {\em Hypercomplex
  analysis}, Trends in Mathematics, pages 165--186. Birkh\"auser Verlag, Basel,
  2009.

\bibitem{cras}
G.~Gentili and D.~C. Struppa.
\newblock A new approach to {C}ullen-regular functions of a quaternionic
  variable.
\newblock {\em C. R. Math. Acad. Sci. Paris}, 342(10):741--744, 2006.

\bibitem{advances}
G.~Gentili and D.~C. Struppa.
\newblock A new theory of regular functions of a quaternionic variable.
\newblock {\em Adv. Math.}, 216(1):279--301, 2007.

\bibitem{heins}
M.~Heins.
\newblock On the {P}hragm\'en-{L}indel\"of principle.
\newblock {\em Trans. Amer. Math. Soc.}, 60:238--244, 1946.

\bibitem{jin1}
Z.~Jin and K.~Lancaster.
\newblock Theorems of {P}hragm\'en-{L}indel\"of type for quasilinear elliptic
  equations.
\newblock {\em J. Reine Angew. Math.}, 514:165--197, 1999.

\bibitem{jin2}
Z.~Jin and K.~Lancaster.
\newblock A {P}hragm\`en-{L}indel\"of theorem and the behavior at infinity of
  solutions of non-hyperbolic equations.
\newblock {\em Pacific J. Math.}, 211(1):101--121, 2003.

\bibitem{kheyfits2}
A.~Kheyfits.
\newblock Phragm\'en-{L}indel\"of principle for {C}lifford-monogenic functions.
\newblock In {\em Complex analysis and its applications}, volume~2 of {\em
  OCAMI Stud.}, pages 237--240. Osaka Munic. Univ. Press, Osaka, 2007.

\bibitem{kheyfits}
A.~Kheyfits and D.~Tepper.
\newblock Subharmonicity of powers of octonion-valued monogenic functions and
  some applications.
\newblock {\em Bull. Belg. Math. Soc. Simon Stevin}, 13(4):609--617, 2006.

\bibitem{levin}
B.~J. Levin.
\newblock {\em Distribution of zeros of entire functions}, volume~5 of {\em
  Translations of Mathematical Monographs}, chapter 1 {\S}14, pages 47--51.
\newblock American Mathematical Society, Providence, R.I., revised edition,
  1980.
\newblock Translated from the Russian by R. P. Boas et al.

\bibitem{massey}
W.~S. Massey.
\newblock {\em A basic course in algebraic topology}, volume 127 of {\em
  Graduate Texts in Mathematics}, chapter V {\S}9, pages 138--140.
\newblock Springer-Verlag, New York, 1991.

\bibitem{phragmen}
E.~Phragm{\'e}n.
\newblock Sur une extension d'un th\'eor\`eme classique de la th\'eorie des
  fonctions.
\newblock {\em Acta Math.}, 28(1):351--368, 1904.

\bibitem{phragmen2}
E.~Phragm{\'e}n and E.~Lindel{\"o}f.
\newblock Sur une extension d'un principe classique de l'analyse et sur
  quelques propri\'et\'es des fonctions monog\`enes dans le voisinage d'un
  point singulier.
\newblock {\em Acta Math.}, 31(1):381--406, 1908.

\end{thebibliography}

\bibliographystyle{abbrv}


\end{document}